\documentclass[12pt]{article}
\usepackage{mathrsfs}
\usepackage{}

\usepackage{authblk}
\usepackage[noadjust]{cite}
\usepackage{amsmath}
\usepackage{amsthm,amsfonts,amssymb}           
\usepackage{bbm}
\usepackage{enumerate}
\usepackage{wasysym}                           
\usepackage{comment}

\usepackage[margin=2.5cm
	       ]{geometry}
\usepackage[pdfstartview=FitH,
            CJKbookmarks=true,
            bookmarksnumbered=true,
            bookmarksopen=true,
            colorlinks,
            linkcolor=blue,
            anchorcolor=blue,
            citecolor=blue,
            urlcolor=blue
            ]{hyperref}

\DeclareFontFamily{OT1}{pzc}{}
\DeclareFontShape{OT1}{pzc}{m}{it}{<-> s * [1.10] pzcmi7t}{}
\DeclareMathAlphabet{\mathpzc}{OT1}{pzc}{m}{it}                 

\newtheorem{theorem}{Theorem}

\newtheorem*{Example}{Example}
\newtheorem{Remark}{Remark}
\newtheorem{Alg}[theorem]{Algorithm}
\newtheorem*{Defn}{Definition}
\newtheorem{lemma}{Lemma}
\newtheorem{corollary}{Corollary}

\newcommand\R{{\mathbb R}}
\newcommand\MP{\mathcal{P}}

\newcommand\MN{\mathcal{N}}
\newcommand\MT{\mathcal{T}}

\newcommand\MPon{\MP_{o}^n}
\newcommand\MTon{\MT_o^n}

\newcommand\cf{C(\R^n)}
\newcommand\cfo{C(\R^n\setminus\{o\})}
\newcommand{\ro}{\R^{n}\setminus \{o\}}
\newcommand\sln{\operatorname{SL}(n)}
\newcommand{\dif}{\mathop{}\!\mathrm{d}}             
\newcommand\relint{\operatorname{relint}}
\newcommand{\ab}[1]{\left(#1\right)}

\newcommand{\ten}[1]{\mathbbm{t}^{}}

\newcommand{\symtenset}[2]{\operatorname{Sym}^{#2}(\R^{#1})}

\makeatletter
\newcommand{\subjclass}[2][1991]{%
  \let\@oldtitle\@title%
  \gdef\@title{\@oldtitle\footnotetext{#1 \emph{Mathematics subject classification.} #2}}%
}
\newcommand{\keywords}[1]{%
  \let\@@oldtitle\@title%
  \gdef\@title{\@@oldtitle\footnotetext{\emph{Key words and phrases.} #1.}}%
}
\makeatother

\title{\bf{$\sln$ contravariant function-valued valuations on polytopes}}
\author[]{Zhongwen Tang \footnote{tzwlxsx@163.com}}
\author[]{Jin Li \footnote{li.jin.math@outlook.com}}
\author[]{Gangsong Leng \footnote{lenggangsong@163.com}}
\affil[]{Department of Mathematics, Shanghai University, Shanghai, China, 200444}
\affil[]{Newtouch Center for Mathematics, Shanghai University, Shanghai, China,  200444}

\date{}

\begin{document}

\maketitle

\begin{abstract}
We present a complete classification of $\sln$ contravariant, $\cfo$-valued valuations on polytopes, without any additional assumptions.
It extends the previous results of the second author [Int. Math. Res. Not. 2020] which have a good connection with the $L_p$ and Orlicz Brunn-Minkowski theory.
Additionally, our results deduce a complete classification of $\sln$ contravariant symmetric-tensor-valued valuations on polytopes.
\end{abstract}

\section{Introduction}

Let $\mathcal{K}^n$ be the set of convex bodies (i.e., compact convex set) in Euclidean space $\mathbb{R}^n$.
In this paper, we always assume $n \ge 3$.
A map $Z$ mapping from a subset of $\mathcal{K}^n$ to an Abelian semigroup $\langle \mathcal{A,+} \rangle$ is called a \emph{valuation} provided
\begin{equation}
Z K+Z L=Z(K\cup L)+Z(K \cap L)\label{def},
\end{equation}
whenever $K,L,K\cup L,K\cap L$ are contained in this subset.
Valuations on polytopes originated in Dehn's solution of Hilbert's Third Problem in 1901 play important roles in convex geometry.
Starting with the celebrated Hadwiger characterization theorem of intrinsic volumes, the theory of valuations has attracted great interest, especially in the last decades; see \cite{alesker1999,haberl2012,haberl2014,haberl2017,klain1995,ludwig2005,ludwig2006,ludwig2010,ludwig2017,ma2021lyz,schuster2018,CLMhadwiger2}.

A \emph{function valued valuation} is a valuation taking values in some function space where addition in $(\ref{def})$ is the ordinary addition of functions.
Previously, the second author established a classification of measurable and $\sln$ contravariant, $C(\R^n)$ valued valuations on polytopes in \cite{li2020} and classification of $\sln$ covariant, $\cfo$ valued valuations on polytopes with some regularity assumptions in \cite{li2021} which extend many results of real-valued valuations, Minkowski valuations, and star body valued valuations and have a good connection with the $L_p$ and Orlicz Brunn-Minkowski theory; see e.g. \cite{lutwak1993,lutwak2000,gardner2014,xi2014,ludwig2005,haberl2017,LYZ00b,CLM2017Min}.
Here $C(\mathbb{R}^n \setminus \{o\})$ is the set of continuous functions $f: \mathbb{R}^n\setminus \{o\} \to \R$ and $C(\mathbb{R}^n)$ is the set of continuous functions $f: \mathbb{R}^n\to \R$.

This paper establishes a complete classification of $\sln$ contravariant $\cfo$ valued valuations on polytopes without any additional assumptions.
The reason that we consider $\cfo$ instead of $\cf$ is that $\ro$ and $\{o\}$ are two different orbits of $\sln$ which may suggest that we do not need the valued functions to be continuous or even defined at the origin.
Our results also cover classifications of real-valued valuations (Ludwig, Reitzner \cite{ludwig2017}), vector-valued valuations (Li, Ma, Wang \cite{li2022}), and symmetric-matrix-valued valuations (Ma, Wang \cite{ma2021lyz}) where measurability are also not imposed.
Here we always assume a polytope is convex.

Let $\MP^n$ be the set of polytopes in $\R^n$ and $\MPon$ be its subset that all polytopes contain the origin.
A function-valued valuation $Z:\mathcal{P}^n\rightarrow C(\mathbb{R}^n\setminus \{o\})$ is called \emph{${\sln}$ contravariant} if
$$Z(\phi P)(x)=ZP(\phi^{-1}x)$$
for every $x\in\ro$, $P\in \mathcal{P}^n$ and $\phi\in \sln$.

Let $h_P$ be the support function of $P \in \MP^n$.
We denote by $V_P$ the cone-volume measure of $P$.
Using the surface area measure $S_P$, the cone-volume measure can be written as $\dif V_{P}=\frac{1}{n} h_P \dif S_{P}$ (Notice that we allow $o \notin P$ so that $V_P$ is a signed measure).
Since the cone-volume measure of a polytope is discrete, we do not distinguish $V_P(u)$ and $V_P(\{u\})$.
Let $\mathcal{N}(P)$ denote the set of all outer unit normals of facets of $P$ and $\mathcal{N}_o(P)$ denote the subset of $\MN(P)$ such that the affine hulls of the corresponding facets do not contain the origin.

Previously known $\cf$ valued $\sln$ contravariant valuations are included in our results.
For instance, $V_1(P,[-x,x])=\frac{1}{n}\sum_{u\in \MN(P)} |x\cdot u|dS_{P}(u)=\frac{2}{n}V_{n-1}(P|x^\bot)|x|$ is the classical projection function of $P$ for $x \in \ro$, and $V_{0}(P)$ is the Euler characteristic, that is, $V_0(P)=1$ if $P$ is not an empty set and $0$ otherwise.
A local version of the Euler characteristic is $(-1)^{\dim P}V_0(o \cap \relint P)$, where $\dim P$ is the dimension of the affine hull of $P$ and $\relint P$ is the relative (with respect to the affine hull of $P$) interior of $P$.
See a local version of the Euler-Schl\"{a}fli-Poincar\'{e} formula in \cite{li2021}.
The new valuation (which may not be measurable) is defined as
\begin{align*}
\Pi_\zeta(P)(x):=\sum\limits_{u\in \mathcal{N}_o(P)}
\zeta\left(\frac{x\cdot u}{h_P(u)},V_P(u)\right), ~P \in \MP^n, ~x \in \R^n
\end{align*}
for some binary function $\zeta:\mathbb{R}^2\rightarrow \mathbb{R}$ satisfying that $\zeta(t,\cdot)$ is a Cauchy function on $\R$ for every $t\in \mathbb{R}$ and $\zeta(\cdot,s)$ is continuous on $\mathbb{R}$ for every $s\in \R$.
Here we say a function $f:\R ~(\text{or~}(0,\infty)) \to \R$ is a Cauchy function if
$f(r+s)=f(r)+f(s)$
for every $r,s\in \R$ (or $r,s >0$, respectively).

Our first result is the following.
\begin{theorem}\label{thm0}
Let $n\geq 3$. A map  $Z:\mathcal{P}_{o}^n \rightarrow C(\mathbb{R}^n  \setminus \{ o \})$ is an ${\sln}$ contravariant valuation if and only if there are constants $c_{n-1},c_0,c_0'\in \mathbb{R}$ and a binary function $\zeta:\mathbb{R}\times (0,\infty)\rightarrow \mathbb{R}$ satisfying that $\zeta$ is Cauchy function on $(0,\infty)$ for every $t\in \mathbb{R}$ and $\zeta(\cdot,s)$ is continuous on $\mathbb{R}$ for every $s>0$ such that
\begin{align*}
ZP(x)   &=\Pi_\zeta(P)(x)+c_{n-1}V_1(P,[-x,x])  +c_0V_0(P)+c_0'(-1)^{\dim P}V_0(o \cap \relint P)
\end{align*}
for every $P\in\mathcal{P}_o^n$ and $x\in \mathbb{R}^n \setminus \{o\}$.
\end{theorem}

We equip $\mathcal{P}^n$ with the Hausdorff metric and equip $C(\mathbb{R}^n  \setminus \{o\})$ (or $C(\mathbb{R}^n)$) the topology induced by uniform convergence on any compact subset in $C(\mathbb{R}^n  \setminus \{o\})$ (or $C(\mathbb{R}^n)$, respectively).
A map on $\mathcal{P}^n$ or $\mathcal{P}_{o}^n$ is (Borel) measurable if the pre-image of every open set in $C(\mathbb{R}^n \setminus \{o\})$ is a Borel set (with respect to the Hausdorff metric).
It is easy to see that $P \to \Pi_\zeta(P)(x)$ is measurable if and only if
\begin{align*}
\Pi_\zeta(P)(x)
=\sum\limits_{u\in \mathcal{N}_o(P)} \eta\ab{\frac{x\cdot u}{h_P(u)}} V_P(u)
\end{align*}
for some continuous function $\eta:\R \to \R$.
That implies the ``only if" part of the following corollary.

\begin{corollary}\label{cor1}
Let $n\geq 3$. A map $Z: \mathcal{P}_{o}^{n} \rightarrow  {C}(\mathbb{R}^n \setminus \{ o \})$ is a measurable, ${\sln}$ contravariant valuation if and only if there are constants $c_0, c_{0}', c_{n-1}\in \mathbb{R}$ and a continuous function $\eta: \mathbb{R} \rightarrow \mathbb{R}$ such that
\begin{equation}\label{eqeta}
\begin{aligned}
ZP(x) &=\sum_{u\in \mathcal{N}_{o}(P)} \eta \ab{  \frac{x\cdot u}{ h_{P}(u)} } V_{P}(u)   \\
&\qquad + c_{n-1} V_{1}(P,[-x,x])
+ c_0 V_{0}(P) +c_{0}' (-1)^{\dim{P}}  V_0(o \cap \relint P)
\end{aligned}
\end{equation}
for every $P\in \mathcal{P}_{o}^{n}$ and $x\in \mathbb{R}^{n} \backslash \{ o  \}$.
\end{corollary}

Then the ``only if" part of Corollary \ref{cor1} implies the ``only if" part of Theorem 1.2 in \cite{li2020}, which is the following.
\begin{corollary}[Li \cite{li2020}]\label{cor0815}
Let $n\ge 3$. If $Z: \mathcal{P}_{o}^{n} \rightarrow  C(\mathbb{R}^n )$ is a measurable, $\sln$ contravariant valuation, then there are constants $c_0, c_{0}', c_{n-1}\in \mathbb{R}$ and a continuous function ${\eta}: \mathbb{R} \rightarrow \mathbb{R}$ such that
\eqref{eqeta} holds for every $P\in \mathcal{P}_{o}^{n}$ and $x\in \mathbb{R}^{n} $.
\end{corollary}

Using our main results, we can also derive many similar corollaries as in \cite{li2020} without assuming the measurability.
However, since those proofs are basically the same, we would like not to repeat such arguments but encourage readers with an interest in this matter to read \cite{li2020} and change the proofs slightly by themselves.
We only derive here the complete classifications of $\sln$ contravariant, symmetric-tensor-valued valuations including previously mentioned vector-valued valuations and symmetric-matrix-valued valuations \cite{li2022,ma2021lyz}.

Let $p \ge 1$ be an integer and $\symtenset{n}{p}$ be the space of symmetric tensors of $p$.
Here a tensor $\ten{p}$ of order $p$ is identified as a multilinear function on $(\R^n)^p$.
A symmetric tensor $\ten{p}$ is identified by its action on $x^p$, denoted by $\langle \ten{p}, x^p \rangle$, for $x\in \R^n$ and $x^p$ denotes the $p$-times tensor product of $x$.
For any $\ten{p} \in \symtenset{n}{p}$ and $\phi \in \sln$, we denote $\phi \ten{p}$ the natural action of $\phi$ to $\ten{p}$.
That is,
$\langle \phi \ten{p}, x^p \rangle =
\langle  \ten{p},  (\phi^t x)^p \rangle$
for any $x \in \R^n$.
We say a map $\mu: \MPon \to \symtenset{n}{p}$ is \emph{$\sln$ contravariant} if
\begin{align*}
\mu(\phi P) = \phi^{-t} \mu P
\end{align*}
for any $\phi \in \sln$;
it is a \emph{symmetric-tensor-valued valuation} if \eqref{def} holds for the addition of tensors.
Now it is easy to see that $\mu: \MPon \to \symtenset{n}{p}$ is an $\sln$ contravariant symmetric-tensor-valued valuation if and only if $Z:\MPon \to \cfo$ defined by
$ZP(x):= \langle   \mu (P) , x^p  \rangle$
for every $P\in \MPon$ and $x\in \ro$, is an $\sln$ contravariant function-valued valuation.

For a Cauchy function $\xi:\R\to\R$, a map $M_{\xi}^{0,p}:\MP^n \to \symtenset{n}{p}$ is defined by
$$M_{\xi}^{0,p}(P):=\sum_{u\in \mathcal{N}_{o}(P)} \ab{\frac{u}{h_{P}(u)}}^p\xi(V_{P}(u)),~P\in \mathcal{P}^n.$$
It extends the tensor $M^{0,p}$ defined in \cite{haberl2017}.

The complete classification of symmetric-tensor-valued valuations on $\MPon$ is the following.
\begin{corollary}\label{1.6}
Let $n \ge 3$ and interger $p \ge 1$.
A map $\mu:\mathcal{P}_{o}^{n}\rightarrow \symtenset{n}{p}$ is an $\sln$ contravariant valuation if and only if there exists a Cauchy function $\xi:(0,\infty) \to \R$ such that
$$\mu(P)=M_{\xi}^{0,p}(P)$$
for every $P\in \mathcal{P}_{o}^n$.
\end{corollary}

Let $[P,o]$ denote the convex hull of $P$ and the origin. We also establish the complete classification of $\sln$ contravariant $\cfo$-valued valuation on $\MP^n$.

\begin{theorem}\label{thm2}
Let $n\geq 3$. A map  $Z:\mathcal{P}^n \rightarrow C(\mathbb{R}^n  \setminus \{ o \})$ is an ${\sln}$ contravariant valuation if and only if there are constants $c_{0}', c_0,\tilde{c}_0, c_{n-1}, \tilde{c}_{n-1} \in \mathbb{R}$ and two binary functions $\zeta_{1},\zeta_{2}:\R^2 \to \R$ satisfying that $\zeta_{1}(t,\cdot)$, $\zeta_{2}(t,\cdot)$ are Cauchy functions on $\R$ for every $t\in \mathbb{R}$ and ${\zeta}_{1}(\cdot,s)$, ${\zeta}_{2}(\cdot,s)$ are continuous on $\mathbb{R}$ for every $s\in \R$ such that
\begin{align*}
ZP(x)
&=\Pi_{\zeta_1}(P)(x)+ \Pi_{\zeta_2}([o,P])(x) + c_{n-1}V_1(P,[-x,x]) +\tilde{c}_{n-1}V_{1}([o,P],[-x,x]) \\
&\quad  +c_0V_0(P) + c_0'(-1)^{\dim P}V_{0}(o \cap \relint P)+ \tilde{c}_{0}  V_{0}( o \cap  P)
\end{align*}
for every $P\in\mathcal{P}_o^n$ and $x\in \mathbb{R}^n \setminus \{o\}$.
\end{theorem}

As in the case of $\MPon$, we can also assume further that valuations are measurable so that
\begin{align*}
\Pi_{\zeta_{1}}(P)+ \Pi_{\zeta_{2}}([o,P])
=\sum\limits_{u\in \mathcal{N}_o(P)} \eta_{1}  \left(\frac{x\cdot u}{h_P(u)}\right) V_P(u)
   + \sum\limits_{u\in \mathcal{N}_o([o,P])} \eta_{2}   \left(\frac{x\cdot u}{h_{[o,P]}(u)}\right) V_{[o,P]}(u)
\end{align*}
for some continuous functions $\eta_1,\eta_2:\R \to \R$.
That concludes the ``only if" part of Theorem 1.4 in \cite{li2020}.

The rest of the paper is organized as follows.
Some notation and background are introduced in the second section.
We prove the map $\Pi_\zeta$ is an $\sln$ contravariant valuation in the third section.
In the last two sections, we prove the main results and corollaries and we remark that Lemma \ref{lm6} is a key improvement over \cite{li2020}.

\section{Notation and Preliminaries}

We refer to Schneider \cite{schneider2014} and Klain \cite{klain1997} as some general references for convex geometry and valuation theory.

We work in Euclidean space $\R^n$ and the inner product of $x,y \in \R^n$ is denoted by $x \cdot y$.
Let $[A_1,...,A_i]$ denote the convex hull of the sets $A_1,...,A_i$ in $\mathbb{R}^n$.
We allow $A_i$ to be a singleton $\{x\}$ and briefly write $x$ instead of $\{x\}$ in this case.

The \emph{support function} of a convex body $K$ is defined by
$h_{K}(x):=\max\{x\cdot y: y\in K   \}$ for $x\in \mathbb{S}^{n-1}$.
It is easy to see that support functions are convex functions and homogeneous of degree 1.
Support functions are important tools in convex geometry since a convex body is identified with its support function. The \emph{Hausdorff distance} of $K$ and $L$ is $\max_{u\in \mathbb{S}^{n-1}}|h_K(u)-h_L(u)|$. Hence, $K_i \rightarrow  K$ with respect to the Hausdorff metric if and only if $h_{K_i} \rightarrow h_K$ uniformly on $\mathbb{S}^{n-1}$.

We call a valuation \emph{simple} if it vanishes on lower-dimensional convex bodies.

Let $P \in \mathcal{P}^n$ and $\eta \subseteq \mathbb{S}^{n-1}$.
The \emph{reverse Gauss map} is
$ {\nu}_{P}^{-1}  (\eta)=\{x\in  P: x\in H_{v,h_{P}(v)} \text{  for some  } v\in \eta     \} \subset \partial P.$
The surface area measure $S_{P}(\eta)=\mathcal{H}^{n-1}(\nu_{P}^{-1}(\eta))$ for any Borel set $\eta \subseteq  \mathbb{S}^{n-1}$, where $\mathcal{H}^{n-1}$ is the $(n-1)$-dimensional Hausdorff measure.
Precisely, $S_{P}=\sum_{u \in \MN(P)}a_{u}\delta_{u}$, where $a_{u}$ is the area of facets of $P$ orthogonal to $u$, and $\delta_{u}$ is the Dirac measure on $\mathbb{S}^{n-1}$ concentrated on the point $u_{i}$.

To prove the ``only if" parts of our main results, we require the following Lemma \ref{lm2} and \ref{lm0.3} in \cite{li2020}.

\begin{lemma}\label{lm2}
Let $Z,Z'$ be two $\sln$ contravariant function-valued valuations on $\mathcal{P}_{o}^n$. If
$Z(sT^d)=Z'(sT^d)$ for every $s>0$ and $0\leq d \leq n$, then $ZP=Z'P$ for every $P\in \mathcal{P}_{o}^n$.
\end{lemma}

\begin{lemma}\label{lm0.3}
Let $Z,Z'$ be two $\sln$ contravariant function-valued valuations on $\mathcal{P}^n$. If
$Z(sT^d)=Z'(sT^d)$ and $Z(s[e_1\dots,e_d])=Z'(s[e_1\dots,e_d])$ for every $s>0$ and $0\leq d \leq n$, then $ZP=Z'P$ for every $P\in \mathcal{P}^n$.
\end{lemma}

\section{The new valuation $\Pi_{\zeta}$.}
We show that $\Pi_{\zeta}$ is a simple and ${\rm SL}(n)$ contravariant $C(\mathbb{R}^{n})$-valued valuation.
Recall that the $\Pi_{\zeta}$ is defined as the following.
\begin{equation*}
	\Pi_{\zeta}P(x):=\sum\limits_{u\in \mathcal{N}_o(P)} \zeta \left( \frac{x\cdot u}{h_P(u)},V_P(u)\right)
	\end{equation*}
for every $P\in \mathcal{P}^{n}$ and $x\in \mathbb{R}^n$, where $\zeta: {\mathbb{R}^2 \rightarrow \mathbb{R}}$ satisfies that $\zeta(t,\cdot)$ is a Cauchy function on $\R$ for every $t\in \mathbb{R}$ and $\zeta(\cdot,s)$ is continuous on $\mathbb{R}$ for $s\in \R$.
\begin{lemma}\label{thm4}
The function-valued map $\Pi_{\zeta}$ is a simple and $\sln$ contravariant, $C(\mathbb{R}^{n})$-valued valuation on $\mathcal{P}^n$.
\end{lemma}
\begin{proof}
Firstly, we prove that $\Pi_{\zeta}$ is a valuation. For $P,Q,P\cup Q \in \mathcal{P}^n$, we divide $\mathbb{S}^{n-1}$ into three parts:
$\omega_1=\{u\in \mathbb{S}^{n-1}:h_{P}(u)=h_{Q}(u)\}$, $\omega_2=\{u\in \mathbb{S}^{n-1}:h_{P}(u)>h_{Q}(u)\}$, $\omega_3=\{u\in \mathbb{S}^{n-1}:h_{P}(u)<h_{Q}(u)\}.$	
For any Borel set $\eta_i\subset \omega_i$, $i=1,2,3$, we have
$$\nu^{-1}_{P\cup Q}(\eta_1)=\nu^{-1}_{P}(\eta_1)\cup \nu^{-1}_{Q}(\eta_1),~\nu^{-1}_{P\cap Q}(\eta_1)=\nu^{-1}_{P}(\eta_1)\cap \nu^{-1}_{Q}(\eta_1),$$
$$\nu^{-1}_{P\cup Q}(\eta_2)=\nu^{-1}_{P}(\eta_2),~\nu^{-1}_{P\cap Q}(\eta_2)=\nu^{-1}_{Q}(\eta_2),$$
$$\nu^{-1}_{P\cup Q}(\eta_3)=\nu^{-1}_{Q}(\eta_3),\nu^{-1}_{P\cap Q}(\eta_3)=\nu^{-1}_{P}(\eta_3),$$
since $h_{P\cap Q}=\min\{h_{P},~h_{Q}\},~h_{P\cup Q}=\max\{h_{P},~h_{Q}\}$.
Then we have
	\begin{align*}
	&V_{P\cup Q}(\eta_1)+V_{P\cap Q}(\eta_1)=V_{P}(\eta_1)+V_{Q}(\eta_2),\\
	&V_{P\cup Q}(\eta_2)=V_{P}(\eta_2),~V_{P\cap Q}(\eta_2)=V_{Q}(\eta_2),\\
	&V_{P\cup Q}(\eta_3)=V_{Q}(\eta_3),~V_{P\cap Q}(\eta_3)=V_{P}(\eta_3),
	\end{align*}
	and
	$$\begin{aligned}
  &\mathcal{N}_o(P\cup Q)\cap \omega_2= \mathcal{N}_o(P)\cap \omega_2,~~
\mathcal{N}_o(P\cap Q)\cap \omega_2= \mathcal{N}_o(Q)\cap \omega_2,\\
	& \mathcal{N}_o(P\cup Q)\cap \omega_3= \mathcal{N}_o(Q)\cap \omega_3,~~
	 \mathcal{N}_o(P\cap Q)\cap \omega_3= \mathcal{N}_o(P)\cap \omega_3.
	\end{aligned}$$
 Thus, we get for $\omega_1$ (note that $\zeta(t,0)=0$) that,
$$\begin{aligned}
	&\sum\limits_{u\in   \mathcal{N}_o(P\cup Q)\cap \omega_1} \zeta \left( \frac{x\cdot u}{h_{P\cup Q}(u)},V_{P\cup Q}(u)\right)
	+\sum\limits_{u\in  \mathcal{N}_o(P\cap Q)\cap \omega_1} {\zeta} \left( \frac{x\cdot u}{h_{P\cap Q}(u)},V_{P\cap Q}(u)\right)
	\\
	&=\sum\limits_{u\in  \mathcal{N}_o(P) \cap\omega_1} {\zeta}\left( \frac{x\cdot u}{h_{P}(u)},V_{P}(u)\right) +\sum\limits_{u\in  \mathcal{N}_o(Q)\cap\omega_1}{\zeta}\left( \frac{x\cdot u}{h_{Q}(u)},V_{Q}(u)\right);\end{aligned}$$
and for $\omega_2$ and $\omega_3$ that
\begin{align*}
&\sum\limits_{u\in  \mathcal{N}_o(P\cup Q) \cap \omega_2} {\zeta} \left( \frac{x\cdot u}{h_{P\cup Q}(u)},V_{P\cup Q}(u)\right)=\sum\limits_{u\in  \mathcal{N}_o(P)\cap \omega_2} {\zeta}  \left( \frac{x\cdot u}{h_{P}(u)},V_{P}(u)\right), \\
&\sum\limits_{u\in  \mathcal{N}_o(P\cap Q)\cap \omega_2} {\zeta} \left( \frac{x\cdot u}{h_{P\cap Q}(u)},V_{P\cap Q}(u)\right)=\sum\limits_{u\in  \mathcal{N}_o(Q)\cap\omega_2} {\zeta} \left( \frac{x\cdot u}{h_{Q}(u)},V_{Q}(u)\right),\\
&\sum\limits_{u\in  \mathcal{N}_o(P\cup Q)\cap \omega_3} {\zeta}  \left( \frac{x\cdot u}{h_{P\cup Q}(u)},V_{P\cup Q}(u)\right)=\sum\limits_{u\in  \mathcal{N}_o(Q)\cap\omega_3}{\zeta} \left( \frac{x\cdot u}{h_{Q}(u)},V_{Q}(u)\right),\\
&\sum\limits_{u\in  \mathcal{N}_o(P\cap Q)\cap \omega_3}  {\zeta}  \left( \frac{x\cdot u}{h_{P\cap Q}(u)},V_{P\cap Q}(u)\right)=\sum\limits_{u\in  \mathcal{N}_o(P)\cap \omega_3} {\zeta}  \left( \frac{x\cdot u}{h_{P}(u)},V_{P}(u)\right).
\end{align*}
All together we get that $\Pi_{\zeta}$ is a valuation.
	
Secondly, let $\phi\in \operatorname{SL}(n), ~P\in\mathcal{P}_o^n$ and $x \in \mathbb{R}^n\setminus \{o\}$. Note that for any $u\in S^{n-1}$ we have $h_{\phi P}(u)=h_{P}(\phi^t u)$ and $V_{\phi P}(u)=V_{P}(\frac{\phi^t u}{|\phi^t u|}) $. Let $v=\frac{\phi^{t}u}{|\phi^{t}u|}$, then we have
	\begin{equation*}
	\begin{aligned}
	\Pi_{{\zeta}}(\phi P)(x)
	&=\sum\limits_{u\in  \mathcal{N}_o(\phi P)} {\zeta} \left( \frac{x\cdot u}{h_{\phi P}(u)},V_{\phi P}(u)\right) =\sum\limits_{u\in  \mathcal{N}_o(\phi P)} {\zeta} \left( \frac{\phi ^{-1}x\cdot \phi ^{t}u}{h_{P}(\phi^t u)},V_{P}\big(\frac{\phi^t u}{|\phi^t u|}\big)\right) \\
	&=\sum\limits_{v\in  \mathcal{N}_o(P)}  {\zeta} \left( \frac{\phi ^{-1}x\cdot v}{h_{P}(v)},V_{P}(v)\right)
=\Pi_{{\zeta}} P(\phi^{-1}x).
	\end{aligned}
	\end{equation*}	
	
Finally, since $\zeta(\cdot,s)$ is continuous for $s\in \R$, we obtain $\Pi_{\zeta}P\in  C(\mathbb{R}^n)$.
The simplicity of $\Pi_\zeta$ follows directly from the definition and $\zeta \left(r,s\right)+\zeta(r,-s)=0$ for every $r,s \in \R$.
\end{proof}

Let $P\in \mathcal{P}^{n}$ and $x\in \mathbb{R}^n$. Define
  $\tilde{\Pi}_{\zeta}P(x):=  \Pi_\zeta([P,o])(x).$

\begin{lemma}\label{lm0.4}
The function-valued map $\tilde{\Pi}_{\zeta}$ is an $\sln$ contravariant $C(\mathbb{R}^n)$-valued valuation on $\mathcal{P}^n$.
\end{lemma}
\begin{proof}
By Lemma \ref{thm4} and the definition of $\tilde{\Pi}_{\zeta}$, it is trivial that $\tilde{\Pi}_{\zeta}P\in C(\mathbb{R}^n)$ for each $P \in \MP^n$.
To show that $\tilde{\Pi}_{\zeta}$ is a valuation, observe that $[P\cup Q,o]=[P,o]\cup [Q,o], [P\cap Q, o]=[P,o]\cap [Q,o]$ when $P \cup Q \in \MP^n$.
Then we have
\begin{align*}
\tilde{\Pi}_{\zeta}(P\cup Q)+\tilde{\Pi}_{\zeta}(P\cap Q)
&=\Pi_\zeta([P\cup Q,o])+ \Pi_\zeta([P\cap Q,o])\\
&=\Pi_\zeta([P,o])+ \Pi_\zeta([Q,o])
=\tilde{\Pi}_{\zeta}(P)+\tilde{\Pi}_{\zeta}(Q)
\end{align*}
for every $P,Q,P\cup Q\in \mathcal{P}^n$.
The $\sln$ contravariance of $\tilde{\Pi}_{\zeta}$ follows similarly from
$[\phi P, o]=\phi[P,o]$ for every $\phi \in {\sln}$ and $P \in \MP^n$.
\end{proof}

\section{Proof of main results }

The ``if" parts of Theorems $\ref{thm0}$ and $\ref{thm2}$ now follow directly from Lemmas $\ref{thm4}$ and $\ref{lm0.4}$.

Let $\mathcal{T}_{o}^n$ be the set of simplices in $\mathbb{R}^n$ with one vertex at the origin.
To prove the ``only if" part of Theorem \ref{thm0}, Lemma \ref{lm2} tells us that we only need to show that $\sln$ contravariant valuations on $\mathcal{T}_{o}^n$ have the corresponding representation.
First, for $\sln$ contravariant map $Z:\mathcal{T}_{o}^{n}  \rightarrow C(\mathbb{R}^n \setminus \{o\})$, we obtain the following.

\begin{lemma}{\label{lm3}}
Let $Z:\mathcal{T}_{o}^n \rightarrow C(\mathbb{R}^n \setminus \{o\})$ be $\sln$ contravariant, $T\in \MTon$ and $x=(x_1,\dots,x_n)^t\in\ro$.
If $\dim T \leq n-2$, then
$$ZT(x)=ZT(e_n),~ x\in \mathbb{R}^n \setminus \{o\};$$
if $\dim T=n-1$ and $T \subset e_n^{\bot}$, then
\begin{align*}
ZT(x)=\begin{cases}
ZT(x_n e_n), &x_n \neq 0, \\
\lim_{t\rightarrow 0} ZT(t e_n), & x_n=0.
\end{cases}
\end{align*}
\end{lemma}

\begin{proof}
Let $T\in \mathcal{T}_{o}^n$ and $\dim T =d<n$. We can assume that the linear hull of $T$ is ${\rm lin}\{ e_1,...,e_d \} $, the linear hull of $\{e_1,...,e_d \}$. Let
$\phi:= \begin{bmatrix} I & A\\ 0 & B  \end{bmatrix} \in {\sln} $, where $I\in \mathbb{R}^{d\times d}$ is the identity matrix, $B\in {\rm SL}(n-d)$, and $0 \in \mathbb{R}^{(n-d)\times d} $ is the zero matrix.
Also, let $x=\begin{bmatrix} x' \\ x'' \end{bmatrix} \in \mathbb{R}^{d \times (n-d)}$. Thus, $\phi T=T$. By the $\sln$ contravariance of $Z$, we have
$$ZT(x)=Z(\phi T)(x)=ZT(\phi^{-1}x)=ZT\begin{pmatrix} x'-AB^{-1}x'' \\ B^{-1}x''  \end{pmatrix},~ \text{since}~ \phi^{-1}= \begin{bmatrix}  I & -AB^{-1} \\ 0 & B^{-1} \end{bmatrix}.   $$

For $d\leq n-2$, assume first $x'' \neq o$.
We can choose a suitable matrix $B$ such that $B^{-1}x''$ is any nonzero vector on ${\rm lin}\{e_{d+1},...,e_n  \} $. After fixing $B$ we can also choose a suitable matrix $A$ such that $x'-AB^{-1}x''$ is any vector in ${\rm lin}\{e_1,...,e_d  \}$. So $ZT(x)$ is a constant function on a  dense set of $\mathbb{R}^n \setminus \{ o\}$. By the continuity of $ZT$, we get $ZT(x)=ZT(e_n).$

For $d=n-1$, $B=1$. If $x_n=x''\neq 0$, we can choose a suitable $A$ such that $x'-AB^{-1}x''=0$.
Hence,
$ZT(x)=ZT(x_n e_n)$.

If $x_n =0$, now we have $x' \neq 0$.
Thus by the continuity of $ZT$,
\begin{align*}
ZT\begin{pmatrix} x' \\ 0    \end{pmatrix}={\lim_{{t}\rightarrow 0}}ZT\begin{pmatrix} x' \\ t \end{pmatrix} = \lim_{t \rightarrow 0}ZT(t e_n).
\end{align*}
\end{proof}

The following lemma characterizes the classic projection functions in the lower-dimensional case.
\begin{lemma} \label{lm4}
If $Z:\mathcal{T}_{o}^{n} \rightarrow C(\mathbb{R}^n \setminus \{o\})$ is an ${\sln}$ contravariant valuation satisfying $Z\{ o \}(e_n)=0$ and $Z[o,e_1](e_n)=0$, then there exists a constant $c_{n-1}\in \mathbb{R}$ such that
$$ZT(x)=c_{n-1}V_{1}(T,[-x,x]), ~ x\in \mathbb{R}^n \setminus \{ o \}$$
for every $T\in \mathcal{T}_{o}^{n}$ satisfying $\dim T \leq n-1$.
\end{lemma}

\begin{proof} For $0<\lambda<1$, let
$H_{\lambda}:=\{ x\in \mathbb{R}^n: x\cdot ((1-\lambda)e_1 - \lambda e_2  )=0   \}$, $H_{\lambda}^{-}:=\{ x\in \mathbb{R}^n: x\cdot ((1-\lambda)e_1 - \lambda e_2  )\leq 0   \}$, $H_{\lambda}^{+}:=\{ x\in \mathbb{R}^n: x\cdot ((1-\lambda)e_1 - \lambda e_2  )\geq 0   \}.$
Since $Z$ is a valuation, then
\begin{equation}
Z(sT^d)(x)+ Z(sT^d \cap H_{\lambda})(x)=Z(sT^d \cap H_{\lambda}^{-} )(x) + Z(sT^d \cap H_{\lambda}^{+})(x), ~x\in \mathbb{R}^n\setminus \{ o \}  \label{a1}
\end{equation}
for $2\leq d \leq n$ and $s>0$. Let $\hat{T}^{d-1}=[o,e_1,e_3,...,e_d]$ and $\phi_{1}, \phi_{2} \in {\sln}$ such that
$\phi_{1} e_{1} = \lambda e_1 +(1-\lambda)e_2, ~\phi_{1}e_{2}=e_{2}, ~\phi_{1}e_{n}=\frac{1}{\lambda}e_n, ~\phi_{1}e_i=e_i, ~{\text{for}}~ 3\leq i\leq n-1, $
and
$\phi_{2}e_1=e_1, ~ \phi_{2}e_2=\lambda e_1 + (1-\lambda)e_2, ~ \phi_{2} e_{n} = \frac{1}{1-\lambda}e_n , ~ \phi_{2}e_{i}=e_{i}, ~{\text{for}}~ 3\leq i \leq n-1. $
For $2\leq d \leq n-1$, we have $$T^d \cap H_{\lambda}^{-}= \phi_{1} T^{d},~~ T^{d}\cap H_{\lambda}^{+}=\phi_{2} T^d, ~~ {\text{and}} ~~ T^d\cap H_{\lambda}=\phi_{1}\hat{T}^{d-1} . $$
Also, since $Z$ is ${\sln}$ contravariant, ($\ref{a1}$) implies that
\begin{equation}
Z(sT^d)(te_n) + Z(s\hat{T}^{d-1})(\lambda t e_n)= Z(sT^d)(\lambda t e_n) + Z(sT^d)((1-\lambda)te_n)   \label{b1}
\end{equation}
for $t\in \mathbb{R}\setminus \{ 0 \}$ and $0< \lambda <1$.
If $d\leq n-2$, now Lemma \ref{lm3} together with the ${\sln}$ contravariance of $Z$ implies
$$Z(s\hat{T}^{d-1})(e_n)=Z(sT^{d-1})(e_n)=Z(sT^d)(e_n).$$
Hence,
$$Z(sT^d)(e_n)=Z(sT^{d-1})(e_n)=\dots =Z(s[o,e_1])(e_n)=Z[o,e_1](e_n).$$
Combined with Lemma {$\ref{lm3}$} and the assumption $Z[o,e_1](e_n)=0$, we have
$Z(sT^d)\equiv 0$
for every $s>0$ and $d\leq n-2$.

If $d=n-1$, the relations $(\ref{b1})$ together with the ${\sln}$ contravariance of $Z$ and the above result for $d\le n-2$ show that
\begin{equation}
ZT^{n-1}(te_n)=ZT^{n-1}(\lambda t e_n) + ZT^{n-1}((1-\lambda)te_n)   \label{d1}
\end{equation}
for $t\in \mathbb{R}\setminus \{0\}$. Let $f(t):=ZT^{n-1}(te_n).$ For arbitrary $t_1,t_2 >0$, setting $t=t_1+t_2$, $\lambda=\frac{t_1}{t_1+t_2}$ in
$(\ref{d1})$, we get that $f$ satisfies the Cauchy functional equation
$f(t_1+t_2)=f(t_1)+f(t_2)$
for every $t_1,t_2>0$. Since $f$ is continuous, there exists a constant $\tilde c_{n-1}\in \mathbb{R}$ such that
$$ZT^{n-1}(te_n)=f(t)= \tilde c_{n-1}t$$
for $t>0$. Also, since $Z$ is ${\sln}$ contravariant, $ZT^{n-1}(te_n)=ZT^{n-1}(-te_n)$, and
$$Z(sT^{n-1})(te_n)=ZT^{n-1}(s^{n-1}te_n)=\left\{
\begin{aligned}
&\tilde c_{n-1}s^{n-1}t,~~ t>0\\
&-\tilde c_{n-1}s^{n-1}t,~~ t<0
\end{aligned}. \right.
$$
Denote $c_{n-1}:=\tilde c_{n-1}\frac{n!}{2}$. Combined with Lemma {$\ref{lm3}$}, we have for $x_n \neq 0$,
\begin{align*}
&Z(sT^{n-1})(x)=Z(sT^{n-1})(x_{n}e_n)=\tilde c_{n-1}s^{n-1}|x_n|=c_{n-1}V_{1}(sT^{n-1},[-x,x]),
\end{align*}
and for $x_n = 0$ (denote $x=\begin{pmatrix} x' \\ 0   \end{pmatrix}$ and $x_t=\begin{pmatrix} x' \\ t    \end{pmatrix}$),
\begin{align*}
Z(sT^{n-1})\begin{pmatrix} x' \\ 0   \end{pmatrix}&= {\lim_{{t}\rightarrow 0}} Z(sT^{n-1})\begin{pmatrix} x' \\ t    \end{pmatrix} \\
&=\lim_{t \rightarrow 0}c_{n-1} V_{1}(sT^{n-1}, [-x_t,x_t])=c_{n-1}V_{1}(sT^{n-1}, [-x,x]).
 \end{align*}

Now the desired result holds for $T=sT^d$ for every $s>0$ and $0\leq d \leq n-1$. Together with the ${\sln}$ contravariance of $Z$, the proof is completed.
\end{proof}

Next, we deal with simple valuations.

\begin{lemma}{\label{lm5}}
If $Z: \mathcal{T}_{o}^{n} \rightarrow C(\mathbb{R}^{n}\setminus \{o\})$ is a simple and ${\sln}$ contravariant valuation, then there is a binary function $\zeta : \mathbb{R}\setminus \{0\} \times (0,\infty) \rightarrow \mathbb{R}$ satisfying that $\zeta(t,\cdot)$ is a Cauchy function on $(0,\infty)$ for every $t\in \mathbb{R} \setminus \{0\}$ and $\zeta(\cdot, s)$ is continuous on $\mathbb{R} \setminus \{0\}$ for every $s>0$ such that
\begin{equation}
Z(sT^n)(te_n)=\zeta\Big(\frac{t}{s}, \frac{s^n}{n!}\Big) =\Pi_{\zeta}(sT^n)(te_n) \label{k1}
\end{equation}
for $s>0$ and $t\in \mathbb{R}\setminus \{0\}$.

\end{lemma}

\begin{proof}
     Let $\phi_{3},\phi_{4}\in {\sln}$ such that
$\phi_{3}e_1=\lambda^{-1/n}(\lambda e_1 + (1-\lambda)e_2),~ \phi_{3}e_2=\lambda^{-1/n}e_2, ~\phi_{3}e_i=\lambda^{-1/n}e_i, ~ \text{for} ~ 3\leq i \leq n, $
and $\phi_{4}e_1=(1-\lambda)^{-1/n}e_1, ~\phi_{4}e_2=(1-\lambda)^{-1/n} (\lambda e_1 + (1-\lambda)e_2), ~\phi_{4}e_i=(1-\lambda)^{-1/n}e_i, ~ \text{for}~3\leq i\leq n.$
Note that $$sT^n \cap H_{\lambda}^{-} = \phi_{3}\lambda^{1/n}sT^n, ~ sT^n \cap H_{\lambda}^{+}=\phi_{4}(1-\lambda)^{1/n}sT^n, ~\text{and}~ sT^n\cap H_{\lambda}=\phi_{3}\lambda^{1/n}s\hat{T}^{n-1}.  $$
The valuation property $(\ref{a1})$ for $d=n$ together with the ${\sln}$ contravariance and simplicity of $Z$ shows that
\begin{equation}
Z(sT^n)(x)=Z(\lambda^{1/n} sT^n)(\phi_{3}^{-1}x)+Z((1-\lambda)^{1/n}sT^n)(\phi_{4}^{-1}x) .\label{e1}
\end{equation}
For $t'\in \mathbb{R}\setminus \{0\}$, choosing $x=t'e_n$ in $({\ref{e1}})$, we have
\begin{equation}
Z(sT^n)(t'e_n)=Z(\lambda^{1/n}sT^n)(\lambda^{1/n}t'e_n)+Z((1-\lambda)^{1/n}sT^n)((1-\lambda)^{1/n}t'e_n) \label{f1}
\end{equation}
for any $0<\lambda<1$ and $s>0$.
Denote
$\zeta(t;\frac{r}{n!}):=Z(r^{1/n}T^n)(r^{1/n}te_n) $
for $r>0$.
Therefore,
$$Z(sT^n)(te_n)=Z(sT^n)\Big(s\frac{t}{s}e_n \Big)=\zeta \Big(\frac{t}{s},\frac{s^n}{n!} \Big).$$
The second equation of the \eqref{k1} follows from
$\zeta \Big(\frac{t}{s},\frac{s^n}{n!}\Big) = \sum_{u\in \mathcal{N}_{o}(sT^n)} \zeta\Big(  \frac{te_n\cdot u}{h_{sT^n}(u)}, V_{sT^n}(u)     \Big)$ and the definition of $\Pi_\zeta$.

Since $t\mapsto Z(sT^n)(te_n)$ is continuous for every $s>0$,
we have $\zeta(\cdot , s)\in C(\mathbb{R}\setminus \{0\})$ for any $s>0$.
For arbitrary $r_1,r_2>0$, $t\in \mathbb{R}\setminus \{0\}$, setting $s=(r_1+r_2)^{\frac{1}{n}}$, $t'=(r_1+r_2)^{\frac{1}{n}}t$, $\lambda=\frac{r_1}{r_1+r_2}$ in $(\ref{f1})$,
we get that $\zeta(t,\cdot)$ is a Cauchy function for every $t\in \mathbb{R}\setminus \{0\}$.
\end{proof}

Previous lemmas are slightly modified versions of the corresponding lemmas in \cite{li2020}.
The following lemma is a key improvement showing that $\zeta$ introduced in Lemma \ref{lm5} can be extended to be defined on $\R \times (0,\infty)$ nicely.
It does not appear in \cite{li2020} since $ZT^n(o)$ is defined therein.

\begin{lemma} \label{lm6}
Let $Z:\mathcal{T}_{o}^n\rightarrow C(\mathbb{R}^n \setminus \{o\})$ be a simple and ${\sln}$ contravariant valuation, and define $\zeta: \mathbb{R}\setminus \{0\} \times (0,\infty) \rightarrow \mathbb{R}$ as in Lemma \ref{lm5}.
Then $\lim_{t\rightarrow 0}\zeta(t,s)$ exists for every $s>0$.

Moreover, define $\zeta(0,s):=\lim_{t\rightarrow 0}\zeta(t,s)$.
Then $\zeta(0,\cdot)$ is a Cauchy function.
\end{lemma}
\begin{proof}
Let $x_1>r>0$ and $x=\lambda^{-1/n}(x_1-r)e_1$, $0<\lambda=\frac{x_1-r}{x_1}<1$, and replace $s$ by $\lambda^{-1/n} s$ in (\ref{e1}).
We obtain
\begin{align}\label{j1}
Z(sT^n)(x_1e_1 - r e_2 ) &= Z\big(\lambda^{-1/n}sT^n\big)\big(\lambda^{-1/n}(x_1-r)e_1 \big) \notag\\
&\qquad -Z\big(\lambda^{-1/n} (1-\lambda)^{1/n}sT^n\big) \big(\lambda^{-1/n}(1-\lambda)^{1/n}(x_1-r)e_1\big).
\end{align}
Then the ${\sln}$ contravariance of $Z$ together with \eqref{k1} and ($\ref{j1}$) implies
\begin{align*}
Z(sT^n)(x_1e_1 -re_2)&=\zeta\Big(\frac{x_1 -r}{s}, \frac{s^n}{\lambda n!} \Big) - \zeta \Big(  \frac{x_1-r}{s}, \frac{(1-\lambda)s^n}{\lambda n!}  \Big)=\zeta \Big(\frac{x_1-r}{s}, \frac{s^n}{n!}    \Big).
\end{align*}
Since $Z(sT^n) \in \cfo$, with $x_1\downarrow r$, we obtain
\begin{equation}
\lim_{t \downarrow 0}\zeta\Big(t, \frac{s^n}{n!} \Big)=Z(sT^n)\big( r(e_1-e_2) \big).\label{j5}
\end{equation}

Let $r>x_1>0$ and $x=(1-\lambda)^{-1/n}(x_1-r)e_2$, $0<\lambda=\frac{x_1}{r}<1$, and replace $s$ by $(1-\lambda)^{-1/n}s$ in $(\ref{e1})$.
We obtain
\begin{align}
Z(sT^n)(x_1e_1 - r e_2 ) &= Z\big((1-\lambda)^{-1/n}sT^n\big)\big((1-\lambda)^{-1/n}(x_1-r)e_2 \big) \notag\\
& \qquad -Z\big(\lambda^{1/n} (1-\lambda)^{-1/n}sT^n\big) \big(\lambda^{1/n}(1-\lambda)^{-1/n}(x_1-r)e_2\big).  \label{j3}
\end{align}
Then the ${\sln}$ contravariance of $Z$ together with \eqref{k1} and ($\ref{j3}$) implies
\begin{align*}
Z(sT^n)(x_1 e_1 -re_2)&=\zeta\Big(\frac{x_1-r}{s}, \frac{s^n}{(1-\lambda)n!}    \Big)  - \zeta \Big( \frac{x_1-r}{s}, \frac{\lambda s^n}{(1-\lambda)n!}  \Big)= \zeta\Big( \frac{x_1-r}{s}, \frac{s^n}{n!}  \Big).
\end{align*}
Let $x_1 \uparrow  r$. We obtain
\begin{equation}
\lim_{t \uparrow 0}\zeta\Big( t, \frac{s^n}{n!} \Big)=Z(sT^n)\big( r(e_1-e_2) \big). \label{j6}
\end{equation}
Hence, $\lim_{t \downarrow 0}\zeta(t, s)=\lim_{t \uparrow 0}\zeta(t, s)$ for every $s>0$ follows directly from $(\ref{j5})$ and $(\ref{j6})$. Now
\begin{align*}
\zeta(0,r+s)=\lim_{t \to 0}\zeta(t,r+s)
=\lim_{t \to 0}\zeta(t,r)+\lim_{t \to 0}\zeta(t,s)
=\zeta(0,r)+\zeta(0,s)
\end{align*}
for every $r,s>0$, which completes the proof.
\end{proof}

We can use the same induction procedure of \cite[Lemma 5.4]{li2020} to get the following Lemma.
The proof is the same hence we omit the proof.

\begin{lemma} \label{lm7}
If $Z:\mathcal{T}_{o}^{n}\rightarrow C(\mathbb{R}^n \setminus \{o\})$ is a simple and ${\sln}$ contravariant valuation and $Z(sT^n)(te_n)=0$ for any $s>0$, $t\in \mathbb{R}^n \setminus \{o\}$, then
$$
Z(sT^n)(x)=0
$$
for any $x\in \mathbb{R}^n \setminus \{o\}$.
\end{lemma}

The ``only if" part of Theorem \ref{thm0} follows directly from Lemma \ref{lm2}, Lemma \ref{thm4}, and the following theorem.
\begin{theorem}\label{thm0aa}
Let $n\geq 3$. If a map  $Z:\mathcal{T}_{o}^n \rightarrow C(\mathbb{R}^n  \setminus \{ o \})$ is an ${\sln}$ contravariant valuation, then there are constants $c_{n-1},c_0,c_0'\in \mathbb{R}$ and a binary function $\zeta:\mathbb{R}\times (0,\infty)\rightarrow \mathbb{R}$ satisfying that $\zeta(t,\cdot)$ is Cauchy function on $(0,\infty)$ for every $t\in \mathbb{R}$ and $\zeta(\cdot,s)$ is continuous on $\mathbb{R}$ for every $s>0$ such that
\begin{align*}
ZT(x)   &=\Pi_\zeta(T)(x)+c_{n-1}V_1(T,[-x,x])+c_0V_0(T)+c_0'(-1)^{\dim T}V_0(o \cap \relint T)
\end{align*}
for every $T\in\mathcal{T}_o^n$ and $x\in \mathbb{R}^n \setminus \{o\}$.
\end{theorem}
\begin{proof}
Let $Z: \mathcal{T}_{o}^{n} \rightarrow  C(\mathbb{R}^n \setminus \{ o \})$ be an ${\sln}$ contravariant valuation. Set $c_0:=Z[0,e_1](e_n)$ and $c_0'=Z\{o\}(e_n)-c_0$.
The new map $Z':\MTon \to \cfo$ defined by
$$Z'T=ZT-c_0 V_0 (T)- c_{0}' (-1)^{\dim T} V_{0}(o\cap \relint T)$$
is an ${\sln}$ contravariant valuation satisfying
$Z'\{o\}(e_n)=0 ~ {\rm and} ~ Z'[o,e_1](e_n)=0.$

By Lemma $\ref{lm4}$, we have
$Z'T(x) = c_{n-1} V_1 (T,[-x,x])$
for every $x\in \mathbb{R}^{n} \backslash \{ o\}$ and $T\in \mathcal{T}_{o}^{n}$ satisfying $\dim T \leq n-1$.
Now set
\begin{equation*}
Z''T(x)=ZT(x)-c_0 V_0(T)-c_0'(-1)^{\dim T} V_{0}(o\cap \relint T)-c_{n-1} V_1(T,[-x,x]) 
\end{equation*}
for every $T\in \mathcal{T}_{o}^n$ and $x\in \mathbb{R}^n \backslash \{ o\}$.
Thus $Z''$ is a simple and ${\sln}$ contravariant valuation.
Then Lemmas {\ref{lm5}} and {\ref{lm6}} show that there is a binary function $\zeta : \mathbb{R} \times (0,\infty) \rightarrow \mathbb{R}$ satisfying that $\zeta(t,\cdot)$ is Cauchy function on $(0,\infty)$ for every $t\in \mathbb{R}$ and $\zeta(\cdot, s)$ is continuous for every $s>0$ such that
\begin{equation*}
Z''(sT^n)(te_n)=\Pi_{\zeta}(sT^n)(te_n)
\end{equation*}
for $s>0$ and $t\in \mathbb{R}\setminus \{0\}$.
By Lemma $\ref{thm4}$, $\Pi_{\zeta}$ is a simple and $\sln$ contravariant valuation on $\MTon$
(Remark that we need $\zeta$ to be defined on $\R \times (0,\infty)$ so that $\Pi_{\zeta}$ is well-defined).
Applying with Lemma {\ref{lm7}} for $Z''-\Pi_\zeta$, we obtain the desired result.
\end{proof}

\begin{proof}[Proof of the ``only if" part of Theorem \ref{thm2}]
Let $Z: \mathcal{P}^n\rightarrow C(\mathbb{R}^n\setminus \{o\})$ be an ${\sln}$ contravariant valuation. From Theorem \ref{thm0aa}, there are constants $a_{n-1},a_{0},a_{0}'$ and a binary function $f_1:\R \times (0,\infty)$ satisfying that $f_{1}(t,\cdot)$ is a Cauchy function on $(0,\infty)$ for every $t\in \mathbb{R}$ and $f_{1}(\cdot,s)$ is continuous on $\mathbb{R}$ for every $s>0$ such that
\begin{equation}\label{21}
\begin{aligned}
&Z(sT^{d})(x) \\
&=\Pi_{f_1}(sT^d)(x)+a_{n-1}V_{1}(sT^d,[-x,x])+a_0V_0(sT^d)+(-1)^d a_0' V_0(o \cap \relint sT^d)
\end{aligned}
\end{equation}
for every $x\in \mathbb{R}^n \setminus\{o\}$.

For any $T\in \mathcal{T}_{o}^n \setminus \{o\}$, we write $T'$ for its facet opposite to the origin.
We define the new valuation $ \check{Z}:\mathcal{T}_{o}^n \rightarrow C(\mathbb{R}^n \setminus \{o\})$ by $\check{Z}(T)=Z(T')$ for any $T\in \mathcal{T}_{o}^n \setminus \{ o\}$ and $ \check{Z}\{o\}=0$.
Since $Z$ is ${\sln}$ contravariant, then $\check{Z}$ is also ${\rm SL}(n)$ contravariant on $\mathcal{T}_{o}^n$. Using Theorem $\ref{thm0aa}$ again, there are constants $b_{n-1},b_0, b_{0}'$ and a binary function $f_2:\R \times (0,\infty)$ satisfying that $f_{2}(t,\cdot)$ is a Cauchy function on $(0,\infty)$ for every $t\in \mathbb{R}$ and $f_{2}(\cdot,s)$ is continuous on $\mathbb{R}$ for every $s>0$ such that
\begin{align}\label{22}
Z(s[e_1,\cdot \cdot \cdot, e_d])(x)
=\check{Z}(sT^d)(x)=\Pi_{{f}_{2}}(sT^d)(x)+b_{n-1}V_{1}(sT^d, [-x,x]) + b_0 V_0(sT^d)
\end{align}
for every $x\in \mathbb{R}^n \setminus \{o\}$, where $(-1)^{d} V_{0}(o\cap \relint(sT^d) )=0$, since $d\geq 1$ in this case.

Now we construct new constants $c_0,c_{0}',\tilde{c}_{0},c_{n-1},\tilde{c}_{n-1}$ and binary functions ${\zeta}_{1}, {\zeta}_{2}:\mathbb{R}\times (0,\infty)\rightarrow\mathbb{R}$ satisfying ${\zeta}_{1}(t,\cdot),{\zeta}_{2}(t,\cdot)$ are Cauchy functions on $(0,\infty)$ for every $t\in \mathbb{R}$ and ${\zeta}_{1}(\cdot,s),{\zeta}_{2}(\cdot,s)$ are continuous for every $s>0$ such that
$$
{\zeta}_{1}(t,s)=f_{1}(t,s)-f_{2}(t,s)+(a_{n-1}-b_{n-1})|t|s, ~~ {\zeta}_{2}(t,s)=f_{2}(t,s)-(a_{n-1}-b_{n-1})|t|s
$$
and $c_{n-1}=a_{n-1}-b_{n-1}$, $\tilde{c}_{n-1}=b_{n-1}$, $c_0=b_0$,
$\tilde{c}_{0}=a_0-b_0$, $c_{0}'=a_{0}'$.

Next, we extend $\zeta_{1}, \zeta_{2}$ to be defined on $\R \to \R$ by
$
\zeta_i(t,-s)=-\zeta_i(t,s)$, and $ \zeta_i(t,0)=0$
for every $t\in \R$, $s>0$ and $i=1,2$.
Therefore, $\zeta_i(t,\cdot)$ is a Cauchy function on $\R$ for every $t \in \R$.
For brevity, we denote
\begin{equation}\label{d3}
\begin{aligned}
AP(x)
&:=\Pi_{\zeta_1}(P)(x)+ \Pi_{\zeta_{2}}([o,P])(x)   + c_{n-1}V_1(P,[-x,x]) +\tilde{c}_{n-1}V_{1}([o,P],[-x,x]) \\
&\quad \quad \quad   +c_0V_0(P) + c_0'(-1)^{\dim P}V_{0}(o \cap \relint P)+ \tilde{c}_{0}  V_{0}( o \cap  P).
\end{aligned}
\end{equation}
Then from the equation (\ref{21}) and (\ref{d3}), we have
$$A({sT^d})(x)=Z(sT^d)(x)$$
for every $x\in \mathbb{R}^n\setminus \{o\}$ and $0\leq d \leq n$.
By $(\ref{22})$, $(\ref{d3})$ and $V_{1}(s[e_1,\cdot \cdot \cdot ,e_d],[-x,x])=\left|  \frac{x\cdot u_{0}}{h_{sT^d}(u_0)}  \right|  V_{n}(sT^d)$, we have
\begin{align*} \label{d4}
A(s[e_1,\cdot \cdot \cdot ,e_d])(x)
&=\Pi_{f_2}(sT^d)(x) -(a_{n-1}-b_{n-1})\left|  \frac{x\cdot u_{0}}{h_{sT^d}(u_0)}  \right|  V_{n}(sT^d) \\
& \qquad+(a_{n-1}-b_{n-1})V_{1}(s[e_1,\cdot \cdot \cdot ,e_d],[-x,x])
+\tilde{c}_{n-1} V_{1}(sT^d, [-x,x]) +c_0  \\
&= Z(s[e_1,\cdot \cdot \cdot,e_d])(x)
\end{align*}
for every $x\in \mathbb{R}^n \setminus \{o\}$, where $u_0=\frac{1}{\sqrt{n}}(e_1+e_2+\cdot \cdot \cdot +e_n)$.
By Lemma \ref{lm4} and Lemma \ref{lm0.4}, $A$ is an $\sln$ contravariant valuation on $\MP^n$.
Together with Lemma $\ref{lm0.3}$, we get $ZP=AP$ for every $P\in \mathcal{P}^n$, which completes the proof of ``only if" part of Theorem $\ref{thm2}$.
\end{proof}

\section{Proof of the corollaries} 

\begin{proof}[Proof of Corollary \ref{cor1} and Corollary \ref{cor0815}.]
Denote by $\tau_1,\tau_2$ be two topologies on $\cf$ induced by the uniform convergence on each compact subset in $\ro$ and the uniform convergence on each compact subset in $\R^n$, respectively.
It is easy to see that $\tau_1 \subset \tau_2$.

Claim 1. The ``only if" part of Corollary \ref{cor1} implies Corollary \ref{cor0815}.

\emph{Proof of Claim 1.} Suppose $Z:\MPon \to \cf$ is a measurable map with respect to $\tau_2$.
Then $Z^{-1}(\beta)$ is a Borel set in $\MPon$ for any $\beta \in \tau_2$.
Let $\tilde \beta$ be an arbitrary open set in $\cfo$.
Notice that $\tilde \beta \cap \cf \in \tau_1 \subset \tau_2$.
Then $Z^{-1} (\tilde \beta)=Z^{-1} (\tilde \beta \cap \cf)$ is also a Borel set in $\MPon$.
Thus $Z:\MPon \to \cfo$ is also measurable. That confirms the claim.

Claim 2. Theorem \ref{thm0} implies the ``only if" part of Corollary \ref{cor1}.

\emph{Proof of Claim 2.} Suppose $Z:\MPon \to \cfo$ is measurable. By Theorem \ref{thm0},  $\Pi_{\zeta}$ must be measurable.
As the proof of \cite[Theorem 1.2]{li2020}, we decompose the function $s \mapsto \Pi_{\zeta}(sT^n)(ste_n)$ as
$$s  \mapsto (ste_n,sT^n) \mapsto (ste_n,Z(sT^n)) \mapsto Z(sT^n)(ste_n)$$
for every fixed $t \in \mathbb{R} \setminus \{0\}$.
Since the first and third maps are continuous and the second map is Borel measurable, the function $s \mapsto \Pi_{\zeta}(sT^n)(ste_n)$ is also Borel measurable.
Notice that we have shown in Lemma \ref{lm5} that $\Pi_{\zeta}(sT^n)(ste_n)=\zeta\ab{t,\frac{s^n}{n!}}$.
It turns out that $\zeta(t,\cdot)$ is measurable for every $t \in \R \setminus \{0\}$.
Since $\zeta(t,\cdot)$ is also a Cauchy function, there is a function $\eta: \R \setminus \{0\} \to \R$ such that $\zeta(t,s)=\eta(t)s$ for every $t\in  \R \setminus \{0\}$ and $s>0$.
Now $\zeta(0,s)=\lim_{t \to 0} \zeta(t,s) = \lim_{t \to 0} \eta(t) s$ implies that $\lim_{t \to 0} \eta(t)$ exists, which is denoted by $\eta(0)$.
That confirms the claim.

Claim 3. the “if” part of Corollary \ref{cor1} is correct, that is, for continuous $\eta:\R \to \R$, the map $\tilde \Pi_\eta:\MPon \to \cfo$ defined by
\begin{equation}\label{eq826}
\tilde \Pi_{\eta}P(x)=\sum\limits_{u\in  \mathcal{N}_o(P)} \eta\ab{\frac{x\cdot u}{h_P(u)}} V_P(u), ~P\in \MPon, ~x\in \ro
\end{equation}
is Borel measurable.

\emph{Proof of Claim 3.}
By allowing $x=o$ in \eqref{eq826}, Theorem 1.2 in \cite{li2020} already shown that $\tilde \Pi_\eta P \in \cf$ for every $P\in \MPon$ and $\tilde \Pi_\eta:\MPon \to \cf$ is Borel measurable with respect to the topology $\tau_2$.
Now let $\beta \in \cfo$ be an open subset and thus $\beta \cap \cf \in \tau_1 \subset \tau_2$.
Then $\tilde \Pi_\eta^{-1}(\beta)=\tilde \Pi_\eta^{-1}(\beta \cap \cf)$ is a Borel set in $\MPon$, which confirms the claim.
\end{proof}

To prove Corollary \ref{1.6}, we first state a stronger result without proof since it is similar to the proofs of Corollaries 2.1 and 2.2 in \cite{li2020}, where we should now use Theorem \ref{thm0} instead of Corollary \ref{cor0815}.
Let real $p\ge 0$. A function $f:\ro  \rightarrow \mathbb{R}$ is called \emph{$p$-homogeneous} if $f(\lambda x)=\lambda^{p} f(x)$ for any $\lambda >0$ and $x\in \mathbb{R}^n$.
For $t\in \R$, we denote $t_+^p:=(\max\{t,0\})^p$ and $t_-^p:=(\max\{-t,0\})^p$.

\begin{corollary} \label{c11}
Let $n \ge 3$ and real $p\geq0$.
A map $Z:\mathcal{P}_{o}^{n}\rightarrow \cfo$ is an ${\sln}$ contravariant valuation satisfying that $ZP$ is $p$-homogeneous for every $P\in \MPon$, if and only if there exist Cauchy functions $\xi_1,\xi_2,\xi_3:(0,\infty)\rightarrow \mathbb{R}$ and constants $c_0,c_0',c_{n-1}\in \mathbb{R}$ such that
\begin{align*}
ZP(x)=\sum\limits_{u\in  \mathcal{N}_o(P)} \ab{\frac{x\cdot u}{h_P(u)}}^p_+ \xi_1(V_P(u)) + \sum\limits_{u\in \mathcal{N}_o (P)}\ab{\frac{x\cdot u}{h_P(u)}}^p_- \xi_2(V_P(u)),
\end{align*}
when $p \notin \{0,1\}$;
\begin{align*}
ZP(x)&=\sum\limits_{u\in  \mathcal{N}_o(P)} \ab{\frac{x\cdot u}{h_P(u)}}_+ \xi_1(V_P(u)) + \sum_{u\in \mathcal{N}_{o}(P)}
 \ab{\frac{x\cdot u}{h_P(u)}}_- \xi_2(V_P(u)) \\
 &\qquad + c_{n-1}V_{1}(P,[-x,x]),
\end{align*}
when $p=1$; and
\begin{align*}
  ZP(x)=\xi_{3}(V_{n}(P)) +c_0V_{0}(P) + c_{0}'(-1)^{\dim P} V_{0}(o\cap \relint P),
\end{align*}
when $p=0$, for every $P\in \mathcal{P}_{o}^n$ and $x\in \mathbb{R}^n\setminus \{o\}$.
\end{corollary}

Now we use Corollary \ref{c11} to prove Corollary \ref{1.6}.
\begin{proof}[Proof of Corollary \ref{1.6}]
Let $p \ge 1$ be an integer.
Since
\begin{align*}
\langle M_{\xi}^{0,p}(P), x^p \rangle=\sum_{u\in  \mathcal{N}_{o}(P)} \ab{\frac{x\cdot u}{h_{P}(u)}}^p \xi\left( V_{P}(u)  \right).
\end{align*}
The ``if" part of Theorem \ref{thm0} implies that $M_{\xi}^{0,p}:\MPon \to \symtenset{n}{p}$ is an $\sln$ contravariant valuation.

Let $\mu:\MPon \to \symtenset{n}{p}$ be an $\sln$ contravariant valuation.
Then $Z:\MPon \to \cfo$ defined by
\begin{align*}
ZP(x):= \langle   \mu (P) , x^p  \rangle, ~P\in\MPon,~x\in \ro,
\end{align*}
is an $\sln$ contravariant valuation such that $ZP$ is $p$-homogeneous for every $P\in \MPon$.
Also
\begin{align}\label{eq826-2}
ZP(-x)=(-1)^pZP(x) .
\end{align}
Applying Corollary \ref{c11} for $P\in \MPon$ with $\dim P <n$, we get $c_{n-1}=0$, otherwise it contradicts \eqref{eq826-2}.
Further applying Corollary \ref{c11} for $P=sT^n$ and $x=e_1$ and $-e_1$.
Together with \eqref{eq826-2}, we have $\xi_1=\xi_2$ for even $p$, and $\xi_1=-\xi_2$ for odd $p$.
For both cases, we have
\begin{align*}
ZP(x)=\sum_{u\in  \mathcal{N}_{o}(P)} \ab{\frac{x\cdot u}{h_{P}(u)}}^p \xi\left( V_{P}(u)  \right),
\end{align*}
for some Cauchy function $\xi:(0,\infty)\to \R$.
Therefore, $ZP=M_{\xi}^{0,p}(P)$ for every $P\in\MPon$.
\end{proof}

\section*{Acknowledgement}
\addcontentsline{toc}{section}{Acknowledgement}
The work of the second author was supported in part by the National Natural Science Foundation of China (12201388).
The work of the first and the third authors was supported in part by the National Natural Science Foundation of
China (12171304).


\end{document}